\documentclass[a4paper,12pt]{amsart}
\pdfoutput=1
\usepackage[T1]{fontenc} 
\usepackage{mathtools}
\usepackage[all]{xy}
\usepackage[english]{babel}
\usepackage{frcursive} 
\usepackage[protrusion=true,expansion=true]{microtype}
\usepackage{eucal}
\usepackage{graphicx}
\usepackage{epsfig}
\usepackage{mathrsfs}
\usepackage{amssymb}
\usepackage{amsxtra}
\usepackage{enumerate}

\newtheorem{theorem}{Theorem}[section]
\newtheorem*{theo}{Theorem}
\newtheorem{proposition}[theorem]{Proposition}
\newtheorem{notation}[theorem]{Notation}
\newtheorem{lemma}[theorem]{Lemma}

\newtheorem{definition}[theorem]{Definition}

\theoremstyle{remark}

\newtheorem{example}[theorem]{\bf Example}

\def \1{\mathbb {1}}
\def \NM{\mathbb{N}}
\def \ZM{\mathbb{Z}}
\def \CM{\mathbb{C}}

\def \QM{\mathbb{Q}}

\def \Der {{\rm Der\,}}

\def \Ham {{\rm Ham\,}}
\def \Aut {{\rm Aut\,}}

\def \p {{\rm exp\,}}

\def \d{\partial}
 
\def\a{\alpha}

\def\p{\varphi}

\def \s{\sigma}

\def \to{\longrightarrow} 
\def \w{\wedge}

\def\del{\nabla}
\def \< {{\langle }}
\def \> {{\rangle }}
\def \( {\left( }
\def \) {\right) }

\newcommand{\Mt}{{\mathcal M}}

\renewcommand{\mod}{{\rm  mod\,}}

\newcommand{\bu}{\bullet}
\newcommand{\bs}{\blacksquare}
\newcommand{\cro}{\times}

\title[The Hamiltonian normal form]{The Hamiltonian normal form}
\author{  Mauricio  Garay and Duco van Straten}
\begin{document}
\maketitle
\begin{abstract}Normal forms given by formal power series are widely used in mathematical physics, although they are often divergent. The first terms of such expansions may provide a good approximation but how to improve the approximation, if the series is divergent? In this paper, we exhibit a globally defined normal form at a stationary point of a Hamiltonian motion. This {\em Hamiltonian normal form} has poles along resonance loci. We show that the Birkhoff normal form arises from it as an asymptotic expansion, thereby confirming Poincaré’s intuition regarding the origin of divergent asymptotic expansions in perturbation theory.
\end{abstract}
\section*{ Introduction}
 In his {\em M\'ethodes mathématiques de la mécanique céleste},   Poincar\'e suggested that perturbative expansions which arise in celestial mechanics
 might be asymptotic expansions of rational functions~\cite[§ 119]{Poincare_Methodes}. He considered the series:
$$s=1+\frac{x}{1+y}+\frac{x^2}{1+2y}+\ldots=\sum_{k=0}^{\infty} \frac{x^k}{1+k y}.$$
 
Assume now that we want to solve the equation
\[0=f(x,y):=-y+\frac{x}{1+y}+\frac{x^2}{1+2 y}+ \frac{x^3}{1+3 y}+\ldots .\]
We may then write $y$ as a series in the variable $x$:
\[ b(x)=x+x^4-2x^5+2x^6+6x^7-35x^8+86x^9+\ldots\]
which we call the {\em Birkhoff expansion} of the {\em generating series} $f$. Both formal series have truncations:
 \begin{align*}
  b_n(x)=x+x^4-2x^5+2x^6+\dots+a_n x^n,\\ f_n(x,y):=-y+\frac{x}{1+y}+\dots+ \frac{x^n}{1+n y} \end{align*}
But these have very different nature. When we compare the graph of the polynomial approximation $b_n$ with the zero locus of the  rational approximations of $f_n$, we see that these have little in common. Both curves are tangent up to order five at the origin but
due to the presence of the poles, the rational approximation defines a curve that intersects the $y$-axis along the poles. The curve defined by $f_n$  exhibits numerous bendings, a phenomenon that does, of course, not occur with the graphs of the polynomial approximations $b_n$:
\ \\
\begin{figure}[htb!]
\includegraphics[width=0.35\linewidth]{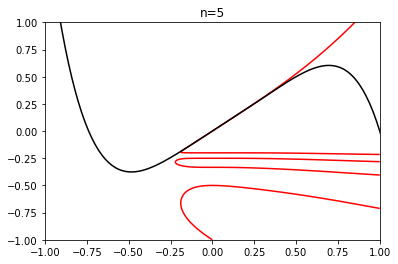}\\
{\em \tiny Comparing the generating series and the Birkhoff series.\\ The graphs of $b_5$ (in black) and the zero locus of $f_5$   (in red).}

 \end{figure}

The series $f$ converges pointwise in the complement of the poles, so it defines a limit set.
 Now, imagine we want to compute the points of this limit set using the polynomial truncations of the $b$-series. This will lead inevitably to values that are far from the actual ones.  
 For instance at $x=1/2$, by truncating the Birkhoff polynomials, we get approximated values: 
 \begin{align*} 0.5,\  0.5&,\ 0.5,\ 0.5625,\ 0.5,\ 0.53125,\ 0.57812,\ 0.441406,\ 0.60937 , 0.581054, \ 0.208,\\ 
\ 1.16064& ,\ -0.00610 ,\ -0.34082, 5.34094,\  -10.59795,
 \ 13.25323,
 \ 18.29940,
 -146.53388\ldots
  \end{align*}
 It would be unreasonable to think that these approximate in any way the points of our curve. At most one can hope that one of the first values gives some reasonable approximation, depending on what reasonable means. In this particular case, we will see that the best approximation at $x=1/2$ is given by $b_6(1/2) \approx 0.53125$. Computing more coefficients of the Birkhoff series is therefore, in this case, irrelevant, but it is not even obvious to guess that $b_6$  is the best truncation a priori. 
 
 By using the Birkhoff series, we end up with at least three problems,  we do not know:
 \begin{enumerate}[{\rm 1)}]
 \item where to stop the computations,
 \item how accurate our result is (unless we run
 into tedious estimates).
 \item how to improve the approximated value.
  \end{enumerate}
  These problems disappear, if we use directly the approximations $f_n$.  
We just substitute the value $x=1/2$ into $f_n$, then reduce to a common denominator. The numerator is then given by polynomials $p_n$:
 \begin{align*}
 p_2&=-2y^2-2y+1,\ p_3=-8y^3 - 12y^2 +y + 3,\ p_4=48y^4 - 88y^3 - 16y^2 + 23y + 7,\\
 p_5&= -384y^5 - 800y^4 - 298y^3 + 163y^2 + 108y + 15.
 \end{align*}
We are now searching for the roots of these polynomials lying near the point $1/2$.
 Let us depict the graphs of the polynomials for $i =2,3,\dots, 6$. We observe that they all pass approximatively through the same root although the convergence is apparently not uniform.\\
\begin{figure}[htb!]
\includegraphics[width=0.35\linewidth]{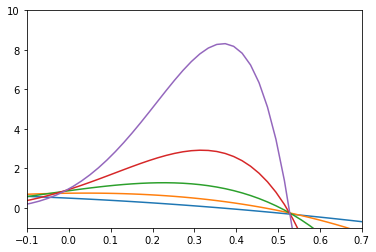}\\
{\em \tiny Approximating a point of the curve $f_n=0$ for $x=1/2$ and $n \leq 5$.}
 \end{figure}
  We solve numerically $p_n=0$  using the Newton method and get a value $y_n$. Then we improve the approximate solution $y_n$ with the next truncation $p_{n+1}=0$ by applying the Newton method with initial value $y_n$. In this way, starting at $y_1=1/2$, we get approximated values which converge:
   \begin{align*}
   0.36602,\  0.46926,\ 0.50593,\ 0.52043,\ 0.52650,\  0.52912,\ 0.53028,\ 0.53080,\ 0.53104,\\
  \ 0.53115, \ 0.53120,\  0.53122
,\ 0.53123
,\ 0.53124
,\ 0.53124
,\ 0.53124
,\ 0.53124
,\ 0.53124.\ 
\end{align*}
  
 This second method is clearly superior to the first one. Nevertheless when performing computations with the Birkhoff normal form in dimension $>1$, it is the former method that is employed. What, then, is the analog of the second computation in the context of a Hamiltonian system? This is the question we aim to answer.
   \begin{theo} Let $H=\sum_{i=1}^d \alpha_i p_i q_i+\dots $  be a non resonant Hamiltonian function, there exists an expansion
$F \in \CM(\omega)[[\tau]],\ \tau_i=p_iq_i$ with poles along the resonance hyperplanes such that:
\begin{enumerate}[{\rm 1)}]
\item the Birkhoff normal form is the Birkhoff series of $F$.
\item there is a Poisson automorphism of $\CM(\omega)[[\tau,q,p]]$ mapping $H$ to its normal form.
\end{enumerate}
\end{theo} 
From the abstract point of view of versal deformations and group actions, this normal form that we call the {\em Hamiltonian normal form} is much more natural than the Birkhoff normal form~\cite{Functors}. It was originally introduced by the first author to prove the Herman invariant tori conjecture~\cite{Herman,Oberwolfach_2012,Herman_conjecture}.  We expect that, in numerical simulations, it might be useful as well and give more accurate approximations than the Birkhoff normal form, just like in the above example. Only time will tell if these expectations actually come true.
  
 \section{\bf The Birkhoff normal form}
\label{BirkhoffNormalForm}
We start by recalling the construction of the Birkhoff normal form from a perspective which will be used later on.
 
\subsection{The symplectic Poisson algebra}
We will be concerned with the structure of an analytic Hamiltonian system 
with $d$ degrees of freedom near a critical point of the form
\[H=\sum_{i=1}^d \alpha_i p_i q_i+\dots\]
where the dots denote higher order terms in the Taylor expansion.
We assume that the {\em frequency vector}:
\[\alpha:=(\alpha_1,\alpha_2,\dots,\alpha_d) \in \CM^d\]
is {\em non-resonant}, i.e., its components $\alpha_i$ are  $\QM$-linearly independent. We are interested in the question which terms appearing in $H$ 
may be transformed away using symplectic coordinate transformations. Rather than coordinate transformations, we found it more practical to deal with automorphisms which is, of course, an equivalent point of view.\\

So we consider the Hamiltonian $H$ as an element of the formal power 
series ring
$$P:=\CM[[q,p]]:=\CM[[q_1,\dots,q_d,p_1,\dots,p_d]].$$ 
The Poisson bracket of $f,g \in P$, defined by
$$\{ f,g \}=\sum_{i=1}^d \d_{q_i}f\d_{p_i}g-  \d_{p_i}f\d_{q_i}g,$$
makes $P$ into a Poisson algebra.
An element $h \in P$ is a power series that can be written as
\[h:=\sum_{a,b} C_{a,b} p^a q^b,\;\;\; C_{a,b}\in \CM, \]
where we use the usual multi-index notation, so that
\[p^aq^b:=p_1^{a_1}p_2^{a_2}\ldots p_d^{a_d} q_1^{b_1}q_2^{b_2}\ldots q_d^{b_d},\] 
and so on. We assign weight $1$ to each of the variables, so 
that the monomial $p^aq^b$ has weight $|a|+|b|$.
We write $h=O(k)$ 
if $h$ only contains monomials of degree $\ge k$, and say that $h$ {\em has
order $k$}.  If $h$ is analytic, it is represented by a convergent series, 
and our usage of the $O$ corresponds to its usual meaning. Algebraically,
the filtration by order is the filtration of $P$ by the powers of the
maximal ideal $\mathcal{M}$:
\[ P \supset \mathcal{M} \supset  \mathcal{M}^2 \supset \mathcal{M}^3 \supset \ldots \supset \{0\},\]
where
\[ \mathcal{M}^k:=\{h \in P\;|\; h = O(k) \}.\]

In a similar way, we can truncate a vector field by truncating its coefficients, but taking the shift of 
grading by $1$ into account (due to the fact the derivative decreases the degree by one). Thus a vector field of order $d$ maps the space of power series of order $k$ to power series of order $d+k$.
 
\begin{notation} If $h$ belongs $P$ or a filtered $P$-module, we denote by
$$\left[ h\right]_i^j $$
for the sum of terms of $h$ of weight (=degree) $\geq i$ and $<j$, 
so that $\left[h\right]_i^{i+1}$
represents the part of $h$ of pure weight $i$. When $j=+\infty$ we omit the letter $j$, when $i=0$ we omit the letter $i$. 
\end{notation}

\begin{definition}
A derivation $v \in Der(P)$ that preserves the Poisson bracket: 
\[ v(\{f,g\})=\{v(f),g\}+\{f,v(g)\}\]
is called a {\em Poisson-derivation} and we denote by $\Theta(P)$ the vector 
space of all Poisson-derivations or Poisson vector fields.
\end{definition}
The map 
\[  P \longrightarrow \Theta(P),\;\;\; h \mapsto \{-,h\}\]
associates to $h$ the corresponding Poisson-derivation, usually called the
{\em Hamiltonian vector field} of $h$. If $h=O(k)$ and $f=O(l)$, then clearly 
$\{f,h\}=O(k+l-2)$, so the vector field $v:=\{-,h\}$ is said to be 
{\em of order $k-2$}, although the coefficients of the vector field $v$ are $O(k-1)$. 
The following is immediate:\\
\begin{lemma} If $h=O(3)$, then one can {\em exponentiate} $v$ and obtain a Poisson 
automorphism of the ring $P$:
\[e^v=Id+\{-,h\}+\frac{1}{2!}\{\{-,h\},h\}+\ldots \in Aut(P).\]
\end{lemma}

If $v$ happens to be analytic then it defines a vector field in a neighbourhood of 
the origin and our derivation $v$ is simply the Lie derivative along this vector field. 
The formal power series $e^v$ is in that case an analytic automorphism and thus defines 
an associated analytic change of variables, the time $=1$ flow of the vector field. 

\subsection{The Birkhoff normal form}
If we let
\[h_0:=\sum_{i=1}^d \alpha_i p_iq_i ,\]
then
\[ \{h_0,p^aq^b\}=(\alpha,a-b) p^aq^b, \]
where $(-,-)$ denotes the standard euclidean scalar product. 
So if $\alpha$ is non-resonant, then each monomial $p^aq^b$ with $a \neq b$
appearing in $H=h_0+O(3)$ can be removed by an application of the
derivation
\[ v_{a,b}=\{-,\frac{1}{(\alpha,a-b)}p^aq^b\} .\]

As the application of $e^{-v_{a,b}}$ to $H$ will remove the term $p^aq^b$ from $H$,
we see that one can construct a sequence of automorphisms 
\[ \varphi_0:=e^{-v_0},\;\;\;\varphi_1:=e^{-v_1},\;\;\;\varphi_2:=e^{-v_2},\ldots \in Aut(P),\]
that remove successively all monomials $p^aq^b$, $a \neq b$ from the Hamiltonian $H$.

To write this iteration more explicitly, let us introduce some notation.
We consider the $\CM$-linear map
$$j:P \to \Theta(P),\ p^aq^b \mapsto \left\{ \begin{matrix} {\displaystyle \{-,\frac{1}{(\alpha,a-b)}p^aq^b\}}  &\text{ if } a \neq b \\
\ \\
0 &\text{ otherwise }\end{matrix} \right. $$
Then we define the iteration by putting $H_0=H$ and
\begin{align*}
v_k&=j([H_k]_{k+3}^{k+4})\\
H_{k+1}&=e^{-v_k}H_k
\end{align*}
so that the automorphism
\[\Phi_k :=\varphi_{k-1}\ldots\varphi_1 \varphi_0,\ \p_i=e^{-v_i}\]
maps $H$ to $H_k$. The infinite composition
\[\Phi:=\ldots \varphi_k \varphi_{k-1}\ldots \varphi_1 \varphi_0 \in Aut(P)\]
is a formal symplectic coordinate transformation that removes all monomials $p^aq^b$, $a \neq b$ from our Hamiltonian $H$, hence we see

\begin{theorem} For any non-resonant $H=h_0+O(3) \in P$ there exists an automorphism
 $\Phi \in Aut(P)$ such that
\[ \Phi(H)=B_H,\]
where $B_H$ is a series of the form
$$B_H:=\sum_{a \in \NM^d}C_ap^aq^a. $$
The series $B_H$ is called the {\em Birkhoff normal form} of $H$. 
\end{theorem}

There exist several variants of this algorithm, differing in details. For example, one may remove several terms at the same time,  which may lead to different normalising transformations $\Phi$; however it is known that different choices lead to the same series $B_H$. 

\begin{example}
Take $d=1$ and consider the Hamiltonian function
$$H(q,p)=pq+p^3+q^3. $$
We determine a sequence of vector fields $v_0, v_1, v_2,\ldots$, where $v_k$ is obtained by removing simultaneously all terms of degree $k+3$. The iteration then begins with
\begin{align*}
H_0&=pq+p^3+q^3\\
v_0&=\{-, 1/3(p^3-q^3)\}\\
H_1&=pq-3p^2q^2+4p^4q+4pq^4+O(6)\\
v_1&=0\\
H_2&=pq - 3q^2p^2 + 4p^4q + 4q^4p - 3/2p^6 - 3/2q^6 - 12p^3q^3 +O(7)\\
v_2&=\{-, 4/3(p^4q-pq^4)\}\\
H_3&= pq- 3p^2q^2  -12p^3q^3 +O(7) \\
\dots
\end{align*}
From this we can read off the first three terms of the Birkhoff normal form, and continuing the process one finds
$$B_H(\tau)=\tau-3\tau^2-12\tau^3-105 \tau^4-1206\tau^5-16002\tau^6-232416\tau^7-3592377\tau^8+o(\tau^8) $$
where $\tau=pq$. (One can show that in this case the inverse power series 
to $B_H(s)$ is a hypergeometric function: $\tau= b \cdot \mbox{}_2F_1(1/3,2/3,1;27 b)$, $b:=B_H(\tau)$.)
\end{example}
\subsection{ The Moser Extension}
\label{SS::Moser}
As the monomials $p_iq_i$ ($i=1,2,\ldots,d$) Poisson commute with the Birkhoff normal form $B_H$, 
Birkhoff normalisation implies that any non-resonant Hamiltonian $H$ is {\em formally completely integrable}. To express this in a manifest way, it is useful to enlarge the ring $P$ and consider
\[Q:=\CM[[\tau,q,p]]=\CM[[\tau_1,\ldots,\tau_d,q_1,\ldots,q_d,p_1,\ldots,p_d]],\] 
with the extra $\tau$-variables, introduced by Moser~\cite{Moser1967}. With the same definition
of the Poisson bracket as before, $Q$  becomes a Poisson algebra with Poisson 
centre $Q_0:=\CM[[\tau]]$. We assign weight $=2$ to the variables $\tau_i$,
so that the $d$ elements
\[ f_i:=p_iq_i-\tau_i \in Q\]
are homogeneous of degree two. These elements Poisson commute, $\{f_i,f_j\}=0$,
and we obtain a Poisson commuting sub-algebra
\[ \CM[[\tau,f]]=\CM[[\tau,f_1,f_2,\ldots,f_d]]=\CM[[\tau,p_1q_1,\ldots,p_dq_d]]\]
containing $Q_0$. The $f_1,f_2,\ldots,f_d$ also generate an ideal\footnote{Here and in the sequel, the notation $\langle f_1,\dots,f_k\rangle$ stands for the ideal generated by elements $f_1,\dots,f_k$.}
\[I=\langle f_1,f_2,\ldots,f_d \rangle  \subset Q=\CM[[\tau,q,p]]\] 
and clearly, the canonical map
\[  \CM[[\tau,p,q]] {\to} \CM[[p,q]],\;\;\;p_i \mapsto p_i,\;q_i\mapsto q_i,\; \tau_i \mapsto p_iq_i .\]
induces an isomorphism of the factor ring $Q/I$ with our original ring $P$:
\[ Q / I \stackrel{\sim}{\to} P . \]
Although $f_i$ maps to zero under this map, the derivation $\{-,f_i\}$ induces
the non-zero derivation $\{-,p_iq_i\}$ on $P$, so the map $Q \to P$ is {\em not}
a Poisson-morphism. The ideal $I^2=\langle f_1,\dots,f_k \rangle^2 \subset Q$ is the square of the ideal $I$, i.e. generated by the elements $f_if_j$, $1 \le i,j\le d$, and plays a very distinguished role in dynamics. 

\begin{lemma} If $T\in I^2$, then $H$ and $H+T$ induce the same Hamiltonian vector 
field on $Q/I=P$. 
\end{lemma}
\begin{proof}
If $T \in I^2$, then $\{h,T\} \subset I$. As a consequence, the difference
between $\{h,H\}$ and $\{h,H+T\}$ belongs to $I$, which is mapped to $0$ in $P$.
\end{proof}

Extending the multi-index notation in an obvious way, we can write
\[ p^aq^a=(\tau+f)^a=\tau^a+\sum_{i=1}^d \partial_{\tau_i} \tau^a f_i+I^2 .\]
The term $\tau^a$ is in the centre of $Q$, whereas the above lemma
implies that  $p^aq^a$ and $\sum_{i=1}^d\d_{\tau_i}\tau^a f_i$ define the 
same derivation on the ring $P=Q/I$.

We can consider the Birkhoff normal form series $B(pq)=B_H(pq)$ as an element 
of $Q$. When we write $pq=\tau+f$, then we find:
\[ B(\tau+f)=B(\tau)+\sum_{i=1}^db_i(\tau)f_i \ \mod I^2 .\]
The formal power series $b_1,\dots,b_d \in \CM[[\tau]]$ are obtained as partial derivatives of $B$, 
considered as a series in the $\tau_i$-variables:
\[b=(b_1,\dots,b_d)=\del B(\tau) .\]

The first term $B(\tau)$ we also call the Birkhoff normal form, written in the
$\tau$-variables. It belongs to the Poisson centre $Q_0$ and is dynamically
trivial, but gets mapped to the non-trivial element $B_H \in P$. The second
term $\sum_{i=1}^db_i(\tau)f_i$ carries the dynamical information in $Q$, but is
mapped by the canonical map $Q \to P$ to zero. 

One has $b(0)=\alpha$, and the higher order terms describe how the frequencies change with $\tau$ and for 
this reason we call it the {\em formal frequency map}.  If the system happens to be integrable, then the series are convergent and the vector $b(\tau)=(b_1(\tau),b_2(\tau),\dots,b_d(\tau))$ is the frequency of motion on the corresponding manifold defined by
$f_i(\tau,q,p)=0$, $i=1,2,\ldots,d$.
\begin{example}
 Take $d=1$, the Hamiltonian
 $$H(q,p)=B(qp)=pq+(pq)^2 $$
 is already in Birkhoff normal form. In the Moser extension we have
 \begin{align*}
 H(q,p) &= \tau+\tau^2+(pq-\tau)+2\tau (pq-\tau)+(pq-\tau)^2\\
   &= B(\tau)+B'(\tau) f+f^2\\
 &=(1+2\tau) pq \ \mod I^2 \oplus \CM[[\tau]] .
 \end{align*}
\end{example}

\section{\bf  The Hamiltonian Normal Form}
\label{HamiltonianNormalForm}
\subsection{Introductory example}
 Consider again the anharmonic oscillator
$$H(q,p)=pq+p^3+q^3. $$
First, we detune the frequencies and consider the function:
$$F_0= (1+\omega)pq+p^3+q^3.$$
The idea is then to take back this function to 
$$A_0=(1+\omega)pq$$
via a Poisson automorphism.
The initialisation of our iteration is therefore
\begin{align*}
A_0&=(1+\omega)pq,\\
F_0&=A_0+p^3+q^3=(1+\omega)pq+p^3+q^3.
\end{align*}
Our first objective is to get rid of the cubic term. This is accomplished by observing that
$$p^3+q^3=\{ A_0,\frac{1}{3 (1+\omega)} (p^3-q^3)\}$$
So we choose
$$v_0=\{ -, \frac{1}{3(1+\omega)} (p^3-q^3)\}.$$
so that the automorphism $\p_0=e^{-v_0}$ transforms $F_0$ into
\begin{align*}
F_1(\tau,\omega,q,p)&=e^{-v_0}(A_0+p^3+q^3)  \\
  &=A_0-v_0(p^3+q^3)+\frac{1}{2!}v_0^2(F_0)+O(6)\\
&=(1+\omega)pq-\frac{3}{1+\omega}q^2p^2+\frac{6}{(1+\omega)^2}(p^4q+q^4p)+O(6).\\
\end{align*}
Now $H$ is recovered from $F_0$ by setting $\omega=0$. As the automorphism $\p_0$ sends the line $\omega=0$ to
itself, so $H=H_0$ is mapped to the restriction of $F_1$ to $\omega=0$:
$$H_1(q,p)= pq-3 q^2p^2+4(p^4q+q^4p)+O(6)$$
So in this way, we got rid of the cubic term in $H_0$. Note that at the first step $A_1=A_0$.
Notice that $F_1$ has poles at $\omega=-1$ which in this case is the only resonance.

Let us now proceed to the next order.
Now we look at the terms of degree 4 and 5. The degree 5 term can be eliminated 
by a Hamiltonian vector field:
$$\{ A_1, \frac{2}{(1+\omega)^3} (p^4q-q^4p)\}=\frac{6}{(1+\omega)^2}(p^4q+q^4p), $$
but something new happens: to suppress the  term $-\frac{3}{1+\omega}q^2p^2$ we need
a non-Hamiltonian Poisson vector field. This is done in two steps. First we note that
$$q^2p^2=(qp-\tau)^2+2\tau qp-\tau^2 .$$
The reason for rewriting the term in this way, is the fact that the terms in the space
$I^2\oplus \CM[[\omega,\tau]]$
{\em do not change the Hamiltonian derivation on the curve $qp=\tau$}.
So we choose 
$$v_1=\{ -, \frac{2}{(1+\omega)^3} (p^4q-q^4p)\}-\frac{6\tau}{1+\omega}\d_\omega $$ 
and get that
\begin{align*}v_1(A_1)&=\frac{6}{(1+\omega)^2}(pq^4+qp^4)-\frac{6\tau}{1+\omega} pq\\
  &=\frac{6}{(1+\omega)^2}(pq^4+qp^4)-\frac{3}{1+\omega} q^2p^2 \ \mod I^2 \oplus \CM[[\omega,\tau]] .
\end{align*}
The difference  $[F_1]_4^6-v_1(A_1)$ is seen to be
\[ \frac{ 6\tau}{1+\omega}pq -\frac{3}{1+\omega}p^2q^2=\frac{3\tau^2}{1+\omega}-\frac{3(pq-\tau)^2}{1+\omega} \in M \subset I^2+\CM[[\omega,\tau]]\]
so we get that
$$F_2=e^{-v_1}F_1= (1+\omega)\tau+\frac{3}{1+\omega}\tau^2+(1+\omega)f +\frac{-3}{1+\omega} f^2+O(6)$$
where $f:=pq-\tau$,.

So the transformation did not bring $F_1$ back to $A_1=A_0$ as there are, like in the Birkhoff normal form, terms which cannot be eliminated by the iterative process. But unlike the Birkhoff normal form, these
residual terms are irrelevant for studying the dynamics!
Now what happens to our function $H_1$? It is mapped to $H_2$, the restriction of $F_2$ to 
$$\p_1(\omega)=0,\ \p_1=e^{-v_1} $$
As $v_1$ contains a non-Hamiltonian term $-6\frac{\tau}{1+\omega}\partial_\omega$, the line $\omega=0$ is
not preserved and, more precisely, we have:
$$\p_1(\omega)=e^{\frac{6\tau}{(1+\omega)}\d_\omega}\omega=\omega+\frac{6\tau}{1+\omega}+O(4) .$$
We get a function of the form indicated by Poincar\'e. If we now compute the first terms of the Birkhoff series, that is, if we solve the truncated equation  
$$\omega+\frac{6\tau}{1+\omega}=0$$
 for $\omega$ we get that:
$$ \omega=-6\tau-36\tau^2+O(6)=-6\tau+O(4)$$
and substitute the result in $F_2$:
$$F_2(\tau,\omega(\tau),q,p)=\tau-3\tau^2+(1-6\tau)f-3f^2+O(6)$$
from which we recover again the first terms of the Birkhoff normal form.

In the next step, we define
\begin{align*}
A_2&=(1+\omega)\tau+\frac{3}{1+\omega}\tau^2+(1+\omega)f +\frac{-3}{1+\omega} f^2\\
   &=(1+\omega)pq +\frac{3}{1+\omega}\tau^2 \;\; \mod I^2 ,
\end{align*}
where $I=(f)$ so that 
\[ F_2=A_2 +O(6).\]
Then we have to look at the terms of degree
$6,7,8,9$ appearing in $F_2$ and determine a vector field $v_2$
\[v_2(A_2) =[F_2]_6^{10} +t , \;\;\; t \in I^2+\CM[[\omega,\tau]] .\]
To see these terms, we have to keep much more terms in the expansions. 
We find
\[A_{H,0}=A_{H,1} =(1+\omega)\tau, \;\;\;\]
\[A_{H,2}=(1+\omega)\tau+\frac{3\tau^2}{1+\omega},\]
\[A_{H,3}=(1+\omega)\tau+\frac{3\tau^2}{1+\omega}+\frac{6\tau^3}{(1+\omega)^3} - \frac{9\tau^4}{(1+\omega)^5}.\]
In this example the denominators appearing are rather simple; in examples
with more variables much more complicated denominators structures arise.
\subsection{The small denominator ring}
 As we saw in our example, we need to consider the Moser variables $\tau$ independently from the frequency variables, which means that we add variables $\omega_1,\dots,\omega_d$.
For a fixed frequency vector $\alpha \in \CM^d$, we define the ring $SD_{\alpha}$ of {\em small denominators} at $\alpha$ as the subring of the field $\CM(\omega)=\CM(\omega_1,\omega_2,\ldots,\omega_d)$ of rational functions, defined by 
localisation of $\CM[\omega]=\CM[\omega_1,\omega_2,\ldots,\omega_d]$ with respect with the multiplicative subset $S$ generated by all linear polynomials $(\a+\omega, J), J \in \ZM^n\setminus \{0\}$:
\[SD_{\alpha}:=\CM[\omega]_{S}:=\CM[\omega,\frac{1}{(\a+\omega,J)}, J \in \ZM^n\setminus\{0\}]
\subset \CM(\omega) .\]
The elements of this ring may have poles along the resonance hyperplanes
$$H_J=\{ \omega \in \CM^d: (\a+\omega,J)=0 \} $$
We now add the $q,p$ variables and define the Poisson algebra:\\
$$R:=SD_{\alpha}[[\tau,p,q]] \subset \CM[[\omega,\tau, p,q]],$$
 in $4d$ variables 
$$\omega_1,\omega_2,\ldots,\omega_d,\tau_1,\ldots,\tau_d,q_1,\ldots,q_d,p_1,\ldots,p_d .$$
The Poisson structure is defined as before by the formula:
$$\{ f,g \}=\sum_{i=1}^d \d_{q_i}f\d_{p_i}g-  \d_{p_i}f\d_{q_i}g.$$ In particular, the Poisson center of
$R$ is the ring
\[R_0:=SD_\a[[\tau]]\subset \CM[[\omega_1,\ldots,\omega_d,\tau_1,\ldots,\tau_d]].\]

Note that the variable $\omega$ has a more {global} character. The relevant filtration of $SD_\a$ is
given by the order of the poles along the resonance hyperplanes.
 
\subsection{The Moser algebra}
The following sub-algebra of $R$ is of importance for our discussion:

\begin{definition} We call the Poisson-commutative algebra
\[M:= R_0+ I^2 \cap R_0[[f]] \]
the {\em Moser-algebra} of $R$  where 
\[ I:=\langle f_1,f_2,\ldots,f_d \rangle  \subset R.\]
\end{definition}
Recall that the space $R_0+I^2$ corresponds to the terms which do not change the Hamiltonian motion. While in standard mechanics the energy is defined "up to a constant" it is in our context defined "up to an element of $R_0+I^2$".  The Moser algebra is a subspace of $R_0+I^2$ which will correspond to
the transversal of Poisson automorphisms acting on our Hamiltonian function.

As before, we denote the vector space of Poisson derivations of the $R$ by 
$\Theta(R)$, which has the structure of a module over the Poisson centre 
$R_0$. One has:

\begin{lemma} The Poisson derivations of $R$ decompose into Hamiltonian and non-exact parts:
$$\Theta(R)=\Ham(R) \oplus \Der_\CM(R_0)  .$$
\end{lemma}
\begin{proof}
We use the notation $x_1,\dots,x_{2d}$ for the variables $q_1,\dots,q_d,p_1,\dots,p_d$ and $y_1,\dots,y_{2d}$ for the variables $\tau_1,\dots,\tau_d,\omega_1,\dots,\omega_d$.

We have $\Der_\CM(R)=  \Der_{R_0}(R_0) \oplus \Der_\CM(R_0)$ which means that any derivation $X \in \Der_\CM(R)$
admits a decomposition
$$X=Y+Z,\ Y=\sum_{i=1}^{2d} a_i \d_{x_i}\  Z=\sum_{i=1}^{2d} b_i \d_{y_i} .$$
Denote by $\pi=\sum_i \d_{q_i} \w \d_{p_i}$ the Poisson bivector. That $X$ is a Poisson vector field means that
$$L_X \pi=L_Y(\pi)=0 $$
Therefore $Y$ preserves the two form $\omega=\sum_i dq_i \w dp_i$ and $Z$ is an arbitrary derivation. By Cartan's formula
$$L_X \omega=d(i_X \omega)=0 $$
where $d=d_{R/R_0}$ is the relative differential:
$$df= \sum_{i=1}^{2d} (\d_{x_i}f) dx_i.$$
Thus we need a relative Poincaré lemma in the ring $R$ in order to conclude that $i_Xd\omega$ is exact and hence $X$ is Hamiltonian. By definition
of $R$ we may decompose a one form $\a$ into homogeneous components with respect to $x$-variables
$$\a=\a_1+\a_2+\dots $$
Denoting by 
$$E=\sum_{i=1}^{2d}x_i \d_{x_i}$$
the Euler vector field. As $L_E \a_i=i\a_i $, by Cartan's formula if $\a$ is closed then
$$d \left( i_E \sum_{i=1}^{2d} \frac{\a_i}{i} \right)=L_E\left(\sum_{i=1}^{2d} \frac{\a_i}{i} \right)=\a $$
where $i_E$ denotes the interior product.
This proves the lemma.
\end{proof}
So an element of $\Theta(R)$ is of the form
\[ v=\{-, h \}+ w\] 
with 
\[ w=\sum_{i=1}^d a_i \frac{\partial}{\partial \omega_i} + b_i \frac{\partial}{\partial \tau_i}, \;\;\;a_i,\;b_i \in R_0.\]

\subsection{{The Hamiltonian normal form iteration}}

\begin{definition} The {\em $\omega$-extension of} $H \in P$ is the element
\[ F= H+\sum_{i=1}^d \omega_i p_i q_i \in R \]
For the $\omega$-extension of $h_0=\sum_{i=1}^d\alpha_ip_iq_i$ we keep a 
special notation:
\[A_0:=\sum_{i=1}^d(\alpha_i+\omega_i) p_iq_i \in R .\]
\end{definition}
So $A_0$ is obtained from $h_0$ by {\em detuning} the frequencies in the most 
general way. One also may interpret it as a versal deformation of $h_0$. 
Starting from a Hamiltonian 

\[ H=\sum_{i=0}^d \alpha_i p_iq_i+O(3),\]
we first form the $\omega$-extension of $H$:
\begin{align*}
F_0&:=H+\sum_{i=1}^d \omega_ip_iq_i\\
   &=A_0+O(3)\\
   &=A_0+[F_0]_3^4+O(4).
   \end{align*}

When we solve a {\em homological equation} of the form
$$v_0(A_0)=[F_0]_3^4 +t_0,\;\;\; t_0 \in M ,$$ 
we obtain a Poisson derivation $v_0$, which we can 
exponentiate to produce an automorphism  $e^{-v_0}$. The application of $e^{-v_0}$ to $F_0$ produces $F_1$, where this term is  removed; we put $A_1=A_0+t_0$. In this particular case, it turns out that $t=0$ but at the next level we have to solve
$$v_1(A_1)=[F_1]_4^6 + t_1 ,\;\;\; t_1 \in M$$ 
for the degree $4$ and $5$ part of $F_1$ on $A_1$ and, as a general rule $t_1 \neq 0$. Then the application of $e^{-v_1}$ to $F_1$ produces $F_2$, where now these terms of degree $4$ and $5$ are removed, but certain terms in the Moser algebra $M$ are introduced. We add these remaining terms  to $A_1$ and obtain $A_2$. Next we solve the homological equation for the terms of degree $6,7,8,9$ of $F_2$, but now on  $A_2$, etcetera. Thus we obtain, by iteration, a sequence of triples
\[ (F_n, A_n, v_n ), \;\;\;n=0,1,2,\ldots\]

\subsection{Ordering the terms of the expansion}
For convenience of the reader we include the following diagram that
indicates the degrees of the quantities that appear in the iteration.\\
\vskip0.1cm
\[
\begin{array}{|c||c|c|c|c|c|c|c|c|c|c|c|c|c|c|c|c|c|}
\hline
 &2&3&4&5&6&7&8&9&10&11&12&13&14&15&16&17&18\\
\hline
\hline
F_0&\bu&\cro&\bs&\bs&\bs&\bs&\bs&\bs&\bs&\bs&\bs&\bs&\bs&\bs&\bs&\bs&\bs\\
\hline
F_1&\bu&\circ&\cro&\cro&\bs&\bs&\bs&\bs&\bs&\bs&\bs&\bs&\bs&\bs&\bs&\bs&\bs\\
\hline
F_2&\bu&\circ&\bu&\circ&\cro&\cro&\cro&\cro&\bs&\bs&\bs&\bs&\bs&\bs&\bs&\bs&\bs\\
\hline
F_3&\bu&\circ&\bu&\circ&\bu&\circ&\bu&\circ&\cro&\cro&\cro&\cro&\cro&\cro&\cro&\cro&\bs\\
\hline
F_4&\bu&\circ&\bu&\circ&\bu&\circ&\bu&\circ&\bu&\circ&\bu&\circ&\bu&\circ&\bu&\circ&\cro\\
\hline
\end{array}
\]\ \\
\vskip0.2cm
The bullets $\bu$ and circles $\circ$ represent terms of $A_n $.   They belong to the Moser-algebra : the $\bu$ terms are constant in columns, the circles  $\circ$ are zero, as the Moser-algebra only has terms of even degree.\\
So $\bu$ and $\circ$ represent the {\em normal form range}, consisting 
of terms of $F_n$ of degree
\[ 2 \le degree <2^n+2\]
The crosses $\times$ represent the terms of $F_n$ that determine the derivations $v_n$. 
These make up what we call the {\em active range} of degrees: 
\[2^n+2 \le degree <2^{n+1}+2 .\] 
The black squares $\bs$ represent the terms of $F_n$ of degree higher
than $2^{n+1}+2$ that do not directly influence the next iteration step, but of course 
must be carried along.\\

We now rewrite the iteration in a form where this trichotomy in degrees 
is manifest. Consider the decomposition
\[ F_n:=A_n+M_n+U_n=\bu+\times+\bs,\]
where
\[A_n:=[F_n]^{2^n+2},\;\;\; M_n:=[F_n]_{2^n+2}^{2^{n+1}+2},\;\;\; U_n:=[F_n]_{2^{n+1}+2},\]
are the lower, middle and upper parts of $F_n$.

If at each step we can solve the linearised equation for $v_n$ then
the Hamiltonian normal form iteration will produce a sequence $(F_n,A_n,v_n)$: 
the series 
\[ F_0=H+\sum_{i=1}^d \omega_i p_iq_i =A_0+O(3)\]
will be transformed by 
\[\Phi_n:=e^{-v_{n-1}}\dots e^{-v_0} \]
to a series of the form
\[F_n=A_{n}+O(2^{n}+2) .\]
If we let $n$ go to $\infty$, we then obtain a formal Poisson automorphism
\[\Phi_{\infty}:=\ldots e^{-v_n}\dots e^{-v_0} \in Aut(R),\]
and obtain
\[F_\infty:=\Phi_{\infty}(F_0)=A_\infty,\;\;\;A_{\infty} \in A_0+M\]
The automorphism $\Phi_{\infty}$ transforms the perturbation $F_0=A_0+O(3)$
back to the normal form $A_0$, plus terms that have no effect on
the dynamics.

\begin{definition} Let $H=\sum_{i=1}^d \alpha_i p_iq_i+\ldots \in P$.
The {\em $k$-th Hamiltonian normal form of $H$} is the series
\[A_{H,k} :=A_k \;\; \mod I \in \CM[[\omega,\tau]],\]
obtained from $A_k$ by the substitution $p_iq_i=\tau_i$.
The {\em Hamiltonian normal form of $H$} is the series
\[ A:=A_H:=A_{\infty}\;\; \mod I  \in \CM[[\omega,\tau]].\]
\end{definition}     
So our aim is to solve the linearised equation, also called the {\em homological equation}.

\subsection{ The homological equation}
\label{SS::Homequation}
We now describe a specific way to solve the homological equation for the $v_k$.
In the algorithm for the Birkhoff normal form the derivations $v_0, v_1, v_2, \ldots$ were determined by applying them to the fixed element $h_0$, whereas 
here the sequence is determined by applying them to elements $A_0, A_1, A_2,\ldots$ which is determined in the iteration process. The infinitesimal action 
$$ \Theta(R) \to R,\;\; v \mapsto v(A_0) $$
on
\[A_0:=\sum_{i=1}^d(\alpha_i+\omega_i) p_iq_i \in R .\]
takes a simple form in the monomial basis:\\
\begin{align*}
\{ A_0, p^a q^b\}&=(\alpha+\omega,a-b)p^aq^b,\\
\d_{\omega_k} A_0&=p_kq_k,\\
\d_{\tau_k} A_0&=0 .\\
\end{align*}

\begin{definition} 
\label{D::L}
We define a $\CM[[\omega,\tau]]$-linear map 
\[L: R \to \Theta(R)=\Ham(R) \oplus \Der(R_0), m \mapsto Lm \]
by setting for $a \neq b$:
\[ Lp^aq^b :=\{-,\frac{1}{(\alpha+\omega,a-b)}p^aq^b\}.\]
For $a=b$, or more generally for a series 
\[m=g(p_1q_1,p_2q_2\,\ldots,p_dq_d)=g(pq)\]
we set
\[Lm:=\sum_{i=1}^d \frac{\partial g(\tau)}{\partial \tau_i}\d_{\omega_i}.\]
\end{definition}

 
\begin{definition} For $A=A_0+T,\ T \in I^2$ we define a linear map
\[ j_{A}: R \to \Theta(R)\]
in terms of $L$ by the formula
\[j_{A}: m \mapsto Lm-L(Lm(T))=L(m-Lm(T))\]
\end{definition}

\begin{proposition} For any $A=A_0+T$, $T \in I^2$ and  any $m \in R$, we have 
\[j_A(m)(A)=m+t,\;\;t \in R_0+I^2,\]
\end{proposition} 
\begin{proof}
First, for $A=A_0$ we have $j_{A_0}=L$.
For $m=p^aq^b$ with $a \neq b$ we have 
\[j_{A_0}(m)(A_0)=\{A_0,\frac{1}{(\alpha+\omega,a-b)}p^aq^b\}=p^aq^b=m\]
and for $m=g(pq)$ we have, with $g_i=\partial_{\tau_i} g$,
\begin{align*}
j_{A_0}(m)(A_0)&=\sum_{i=1}^d g_i(\tau) \frac{\partial A_0}{\partial \omega_i}=\sum_{i=1}^n g_i(\tau)p_iq_i\\
          &=\sum_{i=1}^d g_i(\tau) f_i \ \mod R_0=g(pq)\ \mod R_0+I^2,
\end{align*}
where we used the Taylor expansion
\[g(pq)=g(\tau+f)=g(\tau)+\sum_{i=1}^d g_i(\tau) f_i\ \mod I^2.\]
This shows the correctness for $T=0$. For the general case $A=A_0+T$, we get

\begin{align*}
 j_A(m)(A_0+T)&=Lm(A_0)+Lm(T)-L(Lm(T))A_0-L(Lm(T))(T) \\
              &=m+Lm(T)-Lm(T)-L(Lm(T))(T)\ \mod R_0+I^2\\
              &=m-L(Lm(T))(T)\ \mod R_0+I^2.\\
 \end{align*}
Because $T \in I^2$, it follows that $Lm(T) \in I$. Furthermore,
for any $g \in I$, we have $Lg(T) \in I^2$. This can be seen by writing
$g$ as $\CM[[\omega,\tau]]$-linear combination of terms of the
form $p^aq^b f_i$. If $a \neq b$, $\{T,p^aq^bf_i\} \in I^2$, whereas for
$a=b$, we obtain a combination of terms $\partial_{\omega_i}T$, which is in
$I^2$, as the generators $f_i=p_iq_i-\tau_i$ are independent of $\omega_i$. 
\end{proof}

\subsection{The HNF (Hamiltonian Normal Form) iteration}
With a Hamiltonian $H=\sum_{i=1}^d \alpha_ip_iq_i+O(3) \in P$ as input,
we begin the iteration with the {\em initialisation step} 
\begin{align*}
F_0&=H+\sum_{i=1}^d \omega_ip_iq_i=A_0+O(3)\\
A_0&=\sum_{i=1}^d(\alpha_i+\omega_i)p_iq_i\\
v_0&=j_{A_0}\left( [F_0]_3^4\right).\\
\end{align*}

The next terms are determined by the {\em iteration step}: from $F_n, A_n$ we then obtain
\begin{align*}
F_{n+1}&= e^{-v_{n}}F_{n},\\
A_{n+1}&=A_{n}+\left[F_{n}-v_n(F_n)\right]_{2^{n}+2}^{2^{n+1}+2},\\
v_{n+1}&=j_{A_{n+1}}(\left[F_{n+1}\right]_{2^{n+1}+2}^{2^{n+2}+2})
\end{align*}
It is useful to define the {\em increments}
\[S_{n+1}:=\left[F_{n}-v_n(F_n)\right]_{2^{n}+2}^{2^{n+1}+2},\]
so that:
\[A_{n+1}=A_n+S_{n+1} .\]
There are a few simple but important points to notice:

\begin{proposition}
 \begin{enumerate}[{\rm i)}]
\item The derivation $v_n$ has order $2^n$, i.e. $v_{n}=[v_{n}]_{2^n}$.
\item $F_{n} =A_{n}+O(2^{n}+2)$.
\item $S_n \in M$.
\end{enumerate}
\end{proposition} 
\begin{proof}
  
i) From the recursive definition we see that $v_n$ is obtained by
solving the homological equation with the terms of degrees
${2^n+2}$ up to ${2^{n+1}+2}$ from $F_n$. Taking the Poisson bracket
with a term of degree $2^n+2$ shifts degrees by $2^n$, and similarly for the
non-exact part of $v_n$. So indeed $v_n$ has order $2^n$.\\ 

ii) This follows from an easy induction on $n$. By definition, the statement 
holds for $n=0$. Let us assume that
 $$F_n=A_{n}+O(2^{n}+2)$$ 
From the definition of $F_{n+1}$ we have 
\[F_{n+1}=e^{-v_n}F_n=F_n-v_n(F_n)+\frac{1}{2}v_n^2(F_n)-\ldots \]
and as $v_n$ has order $2^n$, it follows that
\[v_n^2(F_n)=O(2+2^n+2^n)=O(2^{n+1}+2).\]
So we have
\[ F_{n+1}=A_{n}+[F_n-v_n(F_n)]_{2^n+2}^{2^{n+1}+2}+O(2^{n+1}+2)=A_{n+1}+O(2^{n+1}+2).\]
iii) We use induction on $n$ and assume that $S_{n} \in M$. 
From (ii) we have
\[F_n =A_{n}+O(2^{n}+2).\]
The derivation $v_n$ is constructed
to solve the homological equation up to terms of high order:  
\[v_n(A_{n})=[F_n]_{2^n+2}^{2^{n+1}+2}+t+O(2^{n+1}+2),\;\;\; t \in M\]
As we have
\[v_n(F_n)=v_n(A_{n}+O(2^{n}+2))=v_n(A_{n})+O(2^{n+1}+2),\]
we see that the increment
\[[F_n-v_n(F_n)]_{2^n+2}^{2^{n+1}+2} \in  M,\]
hence also $S_{n+1} \in M$.
\end{proof}
Let us denote by $B_n$ the sum of the middle and upper terms, so that
$$F_n=A_n+B_n. $$
As $v_n$ is of order $2^n$, we have
$$ \left[F_{n}-v_n(F_n)\right]_{2^{n}+2}^{2^{n+1}+2}=\left[B_{n}-v_n(A_n)\right]_{2^{n}+2}^{2^{n+1}+2}.$$
Now write
$\tau_n=\left[ - \right]_{2^{n}+2}^{2^{n+1}+2},\ \s_n=\left[ - \right]_{2^{n+1}+2} $,
The iteration is defined by:
\begin{align*}
B_{n+1}&= \s_n(e^{-v_{n}}F_{n}),\\
A_{n+1}&=A_{n}+\tau_n(B_n-v_n(A_n)),\\
v_{n+1}&=j_{A_{n+1}}(\tau_{n+1}(B_{n+1}))
\end{align*}
and is obtained by iterating the maps:\\
\begin{center}
 
\begin{tabular}{| c |}
\hline
\ \\

$\phi_n:(A,B,v) \mapsto (A,0,0)+f_n(A,B,v) $\\
\ \\
$f_n(A,B,v)=\left(\tau_n(B-v(A)), \s_n(e^{-v}(A+B)),j_{A+\tau_n(B-v(A))} \circ \tau_{n+1}(e^{-v}(A+B))\right) $\\
\ \\
\hline
\end{tabular}

\end{center}
\vskip0.5cm
This iteration converges in the $\Mt$-adic topology therefore we get a versal deformation theorem in the Hamiltonian context:
\begin{theorem}Assume that $F=F_0=\sum_{i=1}^d (\alpha_i+\omega_i)p_iq_i+O(3) \in R=SD_\a[[\tau,q,p]]$ is such that $\a$ is non resonant then there
exists a Poisson automorphism $\Phi \in \Aut(R)$ such that
$$\Phi(F) =\sum_{i=1}^d (\a_i+\omega_i)p_iq_i+T,\ T\in M$$
\end{theorem}
 
 \subsection{Relation to the Birkhoff normal form}
\label{SS::relation}
In the Hamiltonian normal form iteration, the original Hamiltonian $H(p,q)$ 
is obtained from $F=F_0$ by equating to zero the functions $\omega_i$:
\[F_0(\omega=0,\tau,p,q)=H(p,q).\]
The automorphism $\Phi \in Aut(SD_\a[[\tau,q,p]])$ which maps $F$ to its normal form maps the function $\omega_i$ 
to a function $G_i(\omega,\tau)$. This function has an expansion similar to that of the Poincar\'e example:
$$G_i(\omega,\tau)=\Phi(\omega_i)=\omega_i +\sum_{I,J} \frac{a^i_{I,J} \tau^I}{(\a+\omega,J)}$$
with poles along the resonance hyperplanes.

By the implicit function theorem, we may solve $G_i=0$ in terms of $\omega_i$ and write $\omega_i=\omega_i(\tau)$ which is the Birkhoff series of the generating 
function $G_i$, according to our terminology. This induces a morphism of Poisson algebras:
$$SD_\a[[\tau,q,p]] \stackrel{\pi}{\to}  \CM[[\tau,q,p]],\ f(\omega,\tau,q,p) \mapsto f(\omega(\tau),\tau,q,p)$$
 
There is a second morphism of Poisson algebras
$$s:\CM[[\tau,q,p]] \to \CM[[q,p]],\  f(\tau,q,p) \mapsto f(qp,q,p)$$
where we substitute $\tau_i$ by $q_ip_i$. Moreover as $ \CM[[q,p]]$ is a Poisson subalgebra of our initial Poisson algebra $SD_\a[[\tau,q,p]]$
this means that the Poisson morphism $\Phi$ induces a symplectic automorphism of $\CM[[q,p]]$. It is defined by substituting $\omega_j$
by $\omega_j(\tau)$ in $\Phi(q_i)$ and $\Phi(p_i)$ and then equating $\tau_j$ with $p_jq_j$. The Birkhoff normal form being unique, we deduce that the image
of the Hamiltonian normal form under the composed map
$$SD_\a[[\tau,q,p]] \stackrel{\Phi}{\to} SD_\a[[\tau,q,p]]  \stackrel{\pi}{\to}  \CM[[\tau,q,p]]  \stackrel{s}{\to}  \CM[[q,p]] $$
is the Birkhoff normal form.
\subsection{A simple example}
\label{SS:Freq_manifold}
Let us go back to our $d=1$ example:
$$H(q,p)=pq+p^3+q^3 $$ 
and the ideal $I=(f)$ with $f=pq-\tau$.  
The iteration produces
$$\begin{array}{ l  l  } 
 A_{H,0}&=(1+\omega)\tau\\
 A_{H,1}&= (1+\omega)\tau\\
   e^{-v_0} \omega &=\omega\\ 
 A_{H,2}&=(1+\omega)\tau+3\frac{\tau^2}{1+\omega}\\
   e^{-v_1} e^{-v_0} \omega&=\omega+\frac{6\tau}{(1+\omega)}+o(2)
\end{array}$$
The generating function $G$ is of the form
$$G(\omega,\tau)= \omega+\frac{6\tau}{(1+\omega)}+o(2)$$
We truncate the generating function $G$ at order $2$:
$$ g_2(\omega,\tau)=\omega+\frac{6\tau}{(1+\omega)}$$
The closure of the curve $\{g_2=0\}$ defines  a parabola:
$$\{ (\tau,\omega):\omega^2+\omega+6\tau=0 \}  $$
\begin{figure}[htb!]
\includegraphics[width=0.4\linewidth]{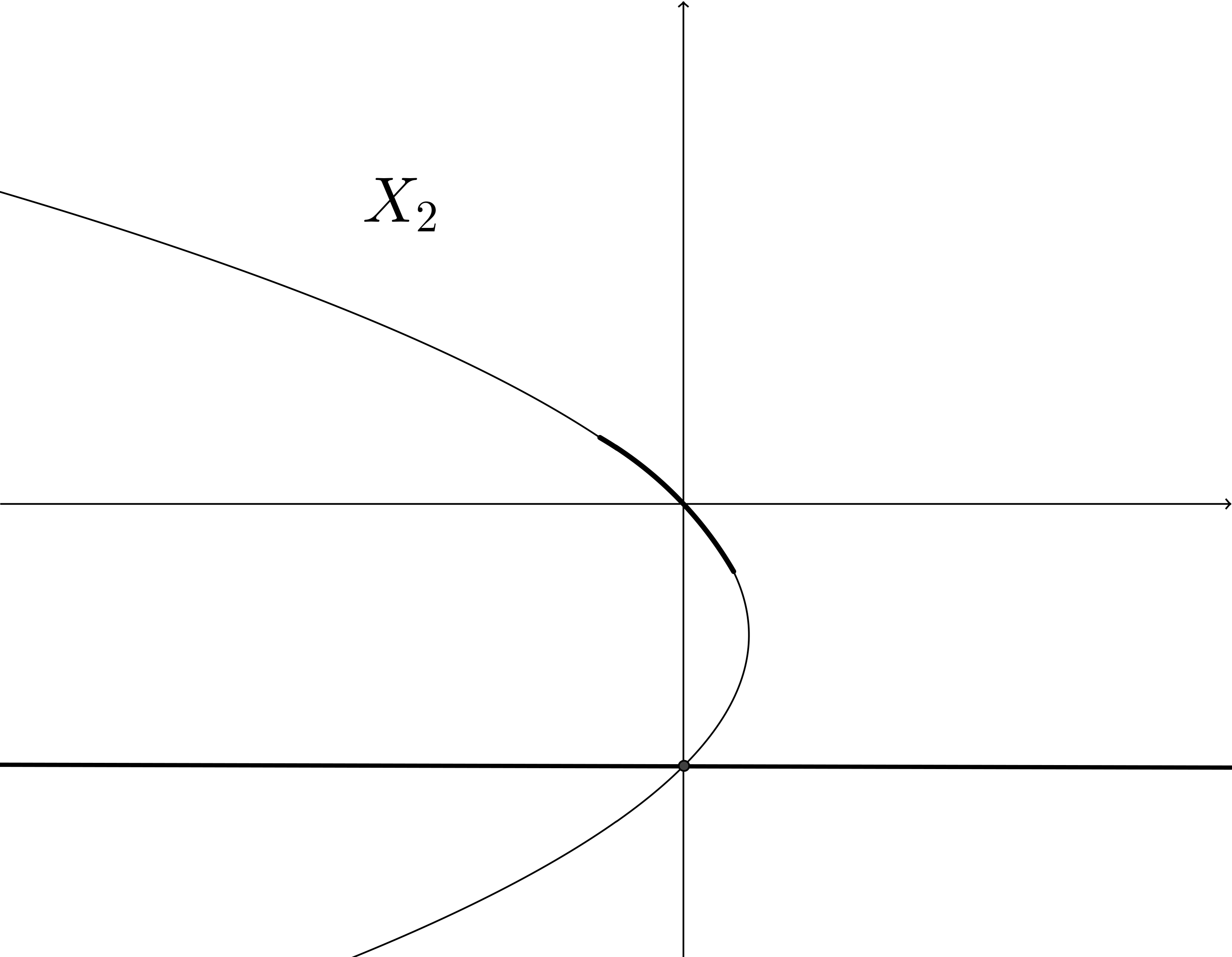} 
\end{figure}
\ \\ 
We now look at the germ of the parabola at the origin, that is, we solve the equation $g_2=0$. This gives the first-order frequency of motion
$$\omega_2(\tau)=-6\tau+O(4) $$ 
Substituting $ \omega_2(\tau)$ into  $F_2$ and taking the constant term
by putting $pq=\tau$, we obtain the first two terms of the Birkhoff normal form:
$$ A_{H,2}(\omega_2(\tau),\tau)=\tau-3\tau^2+O(6)$$

Going to the next order one finds:\\
\[\omega_3(\tau)=-6\tau-36\tau^2-420\tau^3+O(8)\]
\[ A_{H,3}(\omega_3(\tau),\tau)=\tau-3\tau^2-12\tau^3-105\tau^4+O(10),\]
which reproduces the first four terms of the Birkhoff normal form. However the Birkhoff normal form does not see that the curve bends back to the resonance. The functions $H_2, H_3,\ldots$ can be seen as the germs at the origin of $F_2,F_3,\ldots$ restricted to these curves. In this example, the only resonance is at $\omega=-1$ but in higher dimensions resonances have  accumulation points.
 
\bibliographystyle{amsplain}
\bibliography{master}
 \end{document}